\documentclass[a4paper, 11pt, reqno]{amsart}
\usepackage[utf8]{inputenc}
\usepackage[usenames,dvipsnames]{color}
\usepackage{amssymb, amsmath, latexsym, graphics, graphicx,ytableau}

\makeatother

\usepackage[utf8]{inputenc}
\usepackage{amsfonts}
\usepackage{amsthm}
\usepackage{enumerate}

\newtheorem{theorem}{Theorem}[section]
\newtheorem{corollary}[theorem]{Corollary}
\newtheorem{lemma}[theorem]{Lemma}

\newtheorem{question}[theorem]{Question}

\newtheorem{remark}[theorem]{Remark}
\theoremstyle{definition}
\newtheorem{definition}[theorem]{Definition}

\newcommand{\N}{\mathbb{N}}

\newcommand{\T}{\mathbb{T}}

\newcommand{\Pb}{\mathbb{P}}

\newcommand{\cG}{\mathcal{G}}

\newcommand{\floor}[1]{\left \lfloor #1 \right \rfloor }

\title{Semigroup Identities and varieties of Plactic Monoids}
\begin{document}
\date{\today}
\maketitle
\begin{center}
THOMAS AIRD\footnote{This research was supported by a University of Manchester Dean's Scholarship Award. Email \texttt{Thomas.Aird@manchester.ac.uk}.
} \\ \ \\
Department of Mathematics, University of Manchester, \\
Manchester M13 9PL, UK.

\end{center}

\begin{abstract}
We study the semigroup identities satisfied by finite rank plactic monoids. We find a new set of semigroup identities of the plactic monoid of rank $n$ for $n \geq 4$, which are shorter than those previously known when $n \geq 6$. Using these semigroup identities we show that for all $n \in \N$, the plactic monoid of rank $n$ satisfies a semigroup identity not satisfied by the semigroup of $(n+1) \times (n+1)$ upper triangular tropical matrices. We then prove that the plactic monoid of rank $n$ generates a different semigroup variety for each rank $n$.
\end{abstract}

\section{Introduction}
\par The plactic monoids are fundamental algebraic objects which have attracted a lot of interest due to their connections to many different mathematical fields, including representation theory and combinatorics \cite{cain2015finite,hage2017knuth,lopatkin2016cohomology}.
They were first described implicitly by Schensted as a way of finding the maximal length of a non-decreasing subsequence of a given word \cite{Sch1961}. Building on this, Knuth then found a set of defining relations for the plactic monoids, which are now referred to as the Knuth relations \cite{Knuth70}. Subsequently, Lascoux and Sch\"utzenberger \cite{LS81} carried out a systematic study of the plactic monoids. 

\par Studying the semigroup identities satisfied by the plactic monoid of rank $n$, $\Pb_n$, is an active area of research \cite{ATropicalPlactic,NoteOnPlactic,ZurPlacticAlgebra,JKPlactic,KubatOkninskiPlactic3}. While the plactic monoid of rank 1 is a free commutative monoid and hence satisfies the same identities as the free commutative monoids and it can be easily seen that the plactic monoid of rank $2$ satisfies the same identities as the bicyclic monoid; research into the semigroup identities satisfied by the plactic monoid of rank $n$ for $n \geq 3$ has only recently begun. 
This started with Kubat and Okni{\'n}ski \cite{KubatOkninskiPlactic3} showing that the plactic monoid of rank 3 satisfies the semigroup identity $uvvuvu = uvuvvu$ where $u$ and $v$ are the left and right side of Adian's identity respectively. Then, Izhakian \cite{ZurPlacticAlgebra} and Cain, Klein, Kubat, Malheiro, and Okni\'nski \cite{NoteOnPlactic}, found (different) faithful representations of the plactic monoid of rank 3 in $UT_3(\T) \times UT_3(\T)$, thus, showing that $\Pb_3$ satisfies every identity satisfied by $UT_3(\T)$. 
Moreover, in \cite{NoteOnPlactic} the authors showed that the shortest identity satisfied by the plactic monoid of rank $n$ has length greater than $n$. 
Hence, no single semigroup identity is satisfied by every finite rank plactic monoid.
\par Building on these results, Johnson and Kambites \cite{JKPlactic} showed that every finite-rank plactic monoid can be represented by matrices over the tropical semiring. Using these representations they showed that 
every identity satisfied by the plactic monoid of rank $n$ is satisfied by $UT_n(\T)$ and the plactic monoid of rank $n$ satisfies every identity satisfied by $UT_d(\T)$ where $d = \floor{\frac{n^2}{4}+1}$. In particular, they showed that $\Pb_3$ and $UT_3(\T)$ generate the same semigroup variety.

\par The author \cite[Corollary 5.4]{ATropicalPlactic} previously showed that there exists an identity satisfied by the plactic monoid of rank 4 that is not satisfied by $UT_5(\T)$ (and indeed that $UT_n(\T)$ and $UT_{n+1}(\T)$ do not satisfy the same semigroup identities). In this paper, we generalise this result, showing that for all $n \in \N$, there exist semigroup identities satisfied by $\Pb_n$ which are not satisfied by $UT_{n+1}(\T)$. 
Hence, by a result of Johnson and Kambites \cite[Theorem 4.4]{JKPlactic}, we show that for all $n \in \N$, there exists an identity satisfied by $\Pb_n$ which is not satisfied by $\Pb_{n+1}$, and that if the plactic monoid of rank $n$ generates the same variety as $UT_k(\T)$ for some $k$, then $n = k$.
\par It is known that the variety generated by $UT_n(\T)$ is equal to the variety generated by $\Pb_n$ for $n \leq 3$. It remains open if the variety generated by $UT_n(\T)$ is equal to the variety generated by $\Pb_n$ for any $n \geq 4$.
\par Including this introduction, this paper comprises 4 sections. In Section \ref{prelim}, we introduce some notations and definitions that we use throughout. In Section \ref{RepresenationChap}, we recall the tropical representation given in \cite{JKPlactic}, proving several results about the representation and introducing two new functions which we use to prove a technical result. In Section \ref{PlacIdenChap}, we construct a new set of semigroup identities satisfied by the plactic monoid of rank $n$. Using these identities we show that plactic monoids of different rank generate different semigroup varieties.

\section{Preliminaries} \label{prelim}
Given an alphabet $\Sigma$, we denote the monoid of all words over $\Sigma$ under concatenation by $\Sigma^*$ and the subsemigroup of all \emph{non-empty} words over $\Sigma$ by $\Sigma^+$. For $w \in \Sigma^*$, we write $|w|$ for the \emph{length} of $w$ and $w_i$ for the $i$th letter of $w$. For $u,v \in \Sigma^*$ we say $u$ is a \emph{factor} of $v$ if there exist $s,t \in \Sigma^*$ such that $v = sut$, and say $u$ is a \emph{scattered subword} of $v$ if $u = v_{j_1}\cdots v_{j_k}$ for some $1 \leq j_1 < \cdots < j_k \leq |v|$.
\par A (non-trivial) \emph{semigroup identity} is a pair of (distinct) non-empty words, denoted $``u = v"$. For a semigroup $S$, we say that $S$ satisfies the identity $u=v$ if every morphism from $\Sigma^+$ to $S$ maps $u$ and $v$ to the same element of $S$ and define the \emph{variety generated by} $S$ to be the class of all semigroups that satisfy all semigroup identities satisfied by $S$.

\par Let $S$ be a semigroup and $w \in \{a,b\}^+$; then for $x,y \in S$, write $w(a \mapsto x, b \mapsto y)$ to denote the evaluation of $w$ in $S$ obtained by performing the substitution $a \mapsto x$ and $b \mapsto y$. In the case where $S=\Omega^+$ for some alphabet $\Omega$, we write $w[a \mapsto x, b \mapsto y]$, rather than $w(a \mapsto x, b \mapsto y)$, to indicate that $w[a \mapsto x, b \mapsto y]$ is again a word.

\par The \emph{tropical semiring}, denoted $\T$, is the set of real numbers augmented with $-\infty$ defined with two associative binary operations, $a \oplus b = \max(a,b)$ and $ab = a + b$, being the semiring addition and multiplication, respectively. This forms a semiring structure with addition distributing over maximum. 
For $n \in \N$, we define $M_n(\T)$ to be the multiplicative semigroup of all $n \times n$ matrices over $\T$, and $UT_n(\T)$ to be the subsemigroup of all $n \times n$ upper triangular matrices over $\T$, that is, matrices in which every entry below the diagonal is $-\infty$.
\par For a matrix $X \in M_n(\T)$, we write $G_X$ for the \emph{weighted digraph associated to} $X$, that is, the weighted digraph with vertex set $\{1,\dots,n\}$ and edge set $E(G_X)$ containing, 
for all $1 \leq i, j \leq n$ such that $X_{i,j} \neq -\infty$, a directed edge $(i,j)$ with weight $X_{i,j}$.

\par Then, for $X,Y \in M_n(\T)$, we write $G_{X,Y}$ for the \emph{labelled weighted digraph associated to} $X$ \emph{and} $Y$, that is the labelled weighted digraph with vertex set $\{1,\dots,n\}$ and edge set $E(G_X) \sqcup E(G_Y)$ with the edges from $G_X$ labelled by $X$ and the edges from $G_Y$ labelled by $Y$.

\par A \emph{path} $\gamma$ on a digraph is a series of edges $(i_1,j_1),\dots,(i_m,j_m)$ such that $j_k=i_{k+1}$ for all $1 \leq k < m$. We say $g$ is a \emph{node} of $\gamma$ if there exists an edge starting or ending at $g$ in $\gamma$, and call an edge a \emph{loop} if it starts and ends at the same node. 
We say $\delta$ is a \emph{subpath} of $\gamma$ if there exist paths $\delta_1,\delta_2$ such that $\gamma = \delta_1 \circ \delta \circ \delta_2$, and say $\delta$ is a \emph{scattered subpath} of $\gamma$ if there exist paths $\gamma_0,\dots, \gamma_k$ such that $\gamma = \gamma_0 \circ \delta_1 \circ \gamma_1 \circ \cdots \circ \delta_k \circ \gamma_k$ and $\delta = \delta_1 \circ \cdots \circ \delta_k$.

\par A path $\gamma$ is said to have \emph{length} $m$ if $\gamma$ contains $m$ edges (counted with multiplicity), written $|\gamma| = m$. A path is called \emph{simple} if it does not contain any loops. Remark that for any path $\gamma$ there exists a unique maximal-length simple scattered subpath formed by removing all the loops in $\gamma$.
For any word $w \in \{X,Y\}^+$ and $\gamma$ a path in $G_{X,Y}$, we say $\gamma$ is \emph{labelled} by $w$ if $|\gamma| = |w|$ and, for all $1 \leq r \leq |\gamma|$, the $r$th edge of $\gamma$ is labelled by $w_r$, the $r$th letter of $w$. Let $w(\gamma)$ denote the \emph{weight} of $\gamma$, that is, $w(\gamma)$ is the sum of the weights of each edge in $\gamma$.

\par By the definition of $\T$, it can be easily seen that, for $w \in \{a,b\}^*$ and $X,Y \in M_n(\T)$, the $ij$th entry of $w(a \mapsto X, b \mapsto Y)$ is given by the maximum weight of all paths from $i$ to $j$ labelled by the word $w[a \mapsto X,b \mapsto Y]$ in $G_{X,Y}$.
\par For $n \in \N$, we define the plactic monoid of rank $n$, $\Pb_n$, to be the monoid generated by the set $\{1,\dots,n\}$ satisfying the Knuth relations:
\begin{align*}
bca &= bac \text{ for } 1 \leq a < b \leq c \leq n \text{, and } \\ 
cab &= acb \text{ for } 1 \leq a \leq b < c \leq n.
\end{align*}
The plactic monoid of rank $n$ may also be defined combinatorially. A \emph{Young diagram} is a finite left-aligned array of equally-sized boxes in which each row has at most the number of boxes as the box below. Then, we define a \emph{semistandard Young tableau} to be a Young diagram with an element from $\N$ in each box such that the columns are strictly decreasing from top-to-bottom and the rows are weakly increasing from left-to-right.
\par Now, we can define $\Pb_n$ by taking the alphabet $\{1,\dots,n\}$ and identifying all words in $\{1,\dots,n\}^*$ which produce the same semistandard Young tableau under Schensted's insertion algorithm \cite{Sch1961}. For example, the Young tableau
\[\begin{ytableau}
3 \\
1 & 1 & 4
\end{ytableau}\]
can be obtained by applying Schensted's insertion algorithm to the words $1314$, $1341$ or $3114$, all of which can easily be seen to be equal in the plactic monoid by using the Knuth relations.

\section{Plactic Representation} \label{RepresenationChap}
\par We begin this section by introducing some notation. Let $\N = \{1,2,\dots\}$, $[n] = \{1,\dots,n\}$, and $2^{[n]}$ be the power set of $[n]$.
For $S,T \in 2^{[n]}$, we write $S^i$ for the $i$th smallest element of $S$, and say $S \leq T$ if $|S| \geq |T|$ and $S^i \leq T^i$ for each $i \leq |T|$. Then for $S \leq T$, we write $[S,T]$ for the order interval from $S$ to $T$, and $\cup[S,T]$ for the union of all the sets in $[S,T]$. 
\par We can now give the definition of the faithful representation of the plactic monoid of rank $n$ by tropical matrices given in \cite{JKPlactic}.
\begin{definition}
For all $n \in \N$, define the map $\rho_n \colon [n]^* \rightarrow M_{2^{[n]}}(\T)$, where
\[ \rho_n(x)_{P,Q} = 
\begin{cases}
-\infty &\text{if } |P| \neq |Q| \text{ or } P \nleq Q; \\
1 &\text{if } |P| = |Q| \text{ and } x \in \cup[P,Q];\\
0 &\text{otherwise}.
\end{cases}\]
for each generator $x \in \Pb_n$, extending multiplicatively for products of generators and mapping $e$, the identity element of $\Pb_n$, as follows
\[ 
\rho_n(e) =
\begin{cases}
-\infty &\text{if } |P| \neq |Q| \text{ or } P \nleq Q; \\
0 &otherwise.
\end{cases}
\]
\end{definition}

\begin{theorem}[{\cite[Theorem 2.9]{JKPlactic}}]
For all $n \in \N$, $\rho_n \colon \Pb_n \rightarrow M_{2^{[n]}}(\T)$ is a faithful representation of $\Pb_n$.
\end{theorem}

For $S,T \in 2^{[n]}$ and $S \leq T$, we define the \emph{chain length} of $[S,T]$ to be the maximum $t \in \N$ such that $S = P_1 < \cdots < P_t = T$ for some $P_1,\dots,P_t \in [S,T]$.
For $w \in [n]^*$ and $N \subseteq [n]$, let $|w|_N = \sum_{i \in N} |w|_i$, that is, the number of occurrences of the letters from $N$ in $w$.

\begin{lemma} \label{split}
Let $w \in [n]^*$ and $S,T \in 2^{[n]}$ such that $|S| = |T|$ and $S < T$. Then, there exists $N \in [S,T]$ such that $N \neq S$, $|w|_N \geq \min(|w|_S,|w|_T)$ and $[S,N]$ has chain length $\leq n$.
\end{lemma}
\begin{proof}
If the chain length of $[S,T] \leq n$, let $N = T$ and we are done. So suppose the chain length of $[S,T] > n$. Let $S' = S \setminus T$ and $T' = T \setminus S$. Then,
\begin{equation} \label{CompareSTequation}
    |w|_T - |w|_S = (|w|_{T'} + |w|_{S \cap T}) - (|w|_{S'} + |w|_{S \cap T}) = |w|_{T'} - |w|_{S'}.
\end{equation}
Now, as $|S| = |T|$, we have that $|S'| = |T'|$, so let $S' = \{s_1 < \dots < s_{|S'|} \}$ and $T' = \{t_1 < \dots < t_{|S'|} \}$. 
\par Suppose for a contradiction that for all $i$, $|w|_{t_i} - |w|_{s_i} < \min(0,|w|_T-|w|_S)$. Then,
\begin{align*}
    |w|_{T'} - |w|_{S'} =  \sum_{i=1}^{|S'|}(|w|_{t_i} - |w|_{s_i}) &< |S'|\min(0,|w|_T-|w|_S) \\
    &\leq \min(0,|w|_T - |w|_S) 
\end{align*}
contradicting (\ref{CompareSTequation}). So, let $i$ be such that $|w|_{t_i} - |w|_{s_i} \geq \min(0,|w|_T-|w|_S)$ and let $N = (S \setminus \{s_i\}) \cup \{t_i\} $. Then, 
\[|w|_N = |w|_S + |w|_{t_i} - |w|_{s_i} \geq  \min(|w|_{S},|w|_{T}).\]
\par Now, we show that $S \leq N$ and $[S,N]$ has chain length $t_i - s_i +1 \leq n$. Let $x \in S \cap T$, then $x = S^{s} = T^{t}$ for some $s$ and $t$ with $s \geq t$ as $S \leq T$. Define $S_x = S \setminus \{x\}$ and $T_x = T \setminus \{x\}$. Then,
\begin{enumerate}[(i)]
    \item $(S_x)^j = S^j \leq T^j = (T_x)^j$ for $1 \leq j < t$,
    \item $(S_x)^j = S^j < x \leq T^{j+1} = (T_x)^j$ for $t \leq j < s$,
    \item $(S_x)^j = S^{j+1} \leq T^{j+1} = (T_x)^j$ for $s \leq j \leq |S_x|$.
\end{enumerate}
Thus, $S_x \leq T_x$ and hence, by using this reasoning for all $x \in S \cap T$, we have that $S' \leq T'$. Suppose $s_i = S^{l_1}$ and $t_i = T^{l_2}$ for some $i \leq l_1,l_2 \leq |S|$, then by considering (i), (ii), and (iii), we have that $l_1 \leq l_2$.
%Moreover, if $s_i = S^{l_1}$ and $t_i = T^{l_2}$ then, $l_1 \leq l_2$.
\par As $S' < T'$ and $S' \cap T' = \emptyset$, we have that $s_i < t_i$. So, let $l_2 \leq m \leq |S|$ be the greatest such that $S^m < t_i$. Then, we have that
\[ N = \{ S^1 < \cdots < S^{l_1-1} < S^{l_1+1} < \cdots < S^{m} < t_i < S^{m+1} < \cdots <  S^{|S|} \}. \]
Remark that $S^l = N^l$ for $l < l_1$ and $l > m$. So clearly $S \leq N$ and $[S,N]$ has chain length $t_i - s_i +1 \leq n$. 
\par Finally, we aim to show that $N \leq T$. Note that $t_i \leq T^l$ for $l \geq l_2$. Hence $S^{l+1} \leq T^l$ for $l_2 \leq l < m$ and $N^m = t_i \leq T^m$. Thus, $N^l \leq T^l$ for $l_2 \leq l \leq m$. So, we need to show that $N^l = S^{l+1} \leq T^l$ for all $l_1 \leq l < l_2$.
\par For a contradiction, let $l_1 \leq k < l_2$ be the least such that $S^{k+1} > T^k$. Then, if $k \neq l_1$, 
\[ S^{k+1} > T^k > T^{k-1} \geq S^k, \]
and if $k = l_1$,
\[ S^{l_1+1} > T^{l_1} \geq S^{l_1} > S^{l_1-1}. \]
Hence, in either case,  $T^k \notin S \cap T$ as $S^{l_1} \notin S \cap T$.
As $T^k < T^{l_2} = t_i$, we have that $T^k = t_j$ for some $j < i$. Thus, there exist $x_1,\dots,x_{k-j} \in S \cap T$ such that 
\[ x_1 < \dots < x_{l_1-i} < S^{l_1} < x_{l_1-i+1} < \dots < x_{k-j} < T^k < S^{k+1}\]
as $S^{l_1} = s_i$ and $T^k = t_j$. However, by counting the number of $x_s$ and the number of elements of $S$ between $S^{l_1}$ and $S^{k+1}$, we have that $k-j-l_1+i \leq k-l_1$. Thus, $i \leq j$ and we have a contradiction. Therefore, $S^{l+1} \leq T^l$ for all $l_1 \leq l < m$ and $N \in [S,T]$.
\end{proof}

\begin{corollary} \label{splitapply}
Let $w \in [n]^*$ and $S,T \in 2^{[n]}$ such that $|S| = |T|$, and $S < T$. Then,
\begin{enumerate}[(i)]
    \item If $|w|_S \leq |w|_T$, then there exists $N \in [S,T]$ such that $N \neq S$, $[S,N]$ has chain length $\leq n$, and for all $M \in [S,N]$, $|w|_M \leq |w|_N$. \label{splitapplyincre}
    \item If $|w|_T \leq |w|_S$, then there exists $N \in [S,T]$ such that $N \neq T$, $[N,T]$ has chain length $\leq n$, and for all $M \in [N,T]$, $|w|_M \leq |w|_N$. \label{splitapplydecre}
\end{enumerate}
\end{corollary}
\begin{proof}
For case (i), apply Lemma~\ref{split} with $S$ and $T$ to obtain $N' \in [S,T]$ such that $N' \neq S$, $|w|_{N'} \geq |w|_S$, and $[S,N']$ has chain length $\leq n$. Then let $N \in [S,N]$ be minimal (in the partial order) such that $|w|_N \geq |w|_{N'}$ and $N \neq S$.
\par For case (ii), if the chain length of $[S,T] \leq n$, then let $N \in [S,T]$ be maximal (in the partial order) such that $|w|_N \geq |w|_S$ and $N \neq T$ and we are done, so suppose the chain length of $[S,T] > n$. 
We inductively define a sequence $N_1 < \cdots < N_m$ as follows. Let $S = N_1$ and for each $i \geq 1$, let $N_{i+1} \in [N_i,T]$ be obtained by applying Lemma~\ref{split} to $N_i$ and $T$. 
As the chain length of $[N_{i+1},T]$ is strictly less than the chain length of $[N_i,T]$, we may repeat until there exists $m \in \N$ such that the chain length of $[N_m,T]$ is at most $n$. Note that $|w|_{N_m} \geq |w|_T$ and $N_m \neq T$ as the chain length of $[N_{m-1},N_m]$ is at most $n$. Then, let $N \in [N_m,T]$ be maximal (in the partial order) such that $|w|_N \geq |w|_{N_m}$ and $N \neq T$.
\end{proof}

% For $S,T \in 2^{[n]}$, we say that $S \wedge T$ and $S \vee T$ \emph{exist}, if there exist $S \wedge T$, $S \vee T \in 2^{[n]}$ such that $S \wedge T \leq S,T \leq S \vee T$, and if $N \leq S,T \leq M$, then $N \leq S \wedge T$ and $S \vee T \leq M$. We now proceed by showing that for every $S,T \in 2^{[n]}$, $S \wedge T$ and $S \vee T$ exist.

\begin{lemma}
$2^{[n]}$ is a lattice under $\leq$ where the meet and join operations are defined by
\begin{align*}
        S \wedge T &= \bigcup_{i=1}^k\min(S^i,T^i) \cup
    \begin{cases}
        \{S^{k+1},\dots,S^{|S|}\} &\text{if }|S| \geq |T| \\
        \{T^{k+1},\dots,T^{|T|}\} &\text{if }|S| < |T| \\
    \end{cases} \\
    S \vee T &= \bigcup_{i=1}^k\max(S^i,T^i)
\end{align*}
where $k = \min(|S|,|T|)$.
\end{lemma}
% \begin{lemma}
% Let $S,T \in 2^{[n]}$, then $S \wedge T$ and $S \vee T$ exist and
% \begin{align*}
%         S \wedge T &= \bigcup_{i=1}^k\min(S^i,T^i) \cup
%     \begin{cases}
%         \{S^{k+1},\dots,S^{|S|}\} &\text{if }|S| \geq |T| \\
%         \{T^{k+1},\dots,T^{|T|}\} &\text{if }|S| \leq |T| \\
%     \end{cases} \\
%     S \vee T &= \bigcup_{i=1}^k\max(S^i,T^i)
% \end{align*}
% where $k = \min(|S|,|T|)$.
% \end{lemma}
\begin{proof}
Without loss of generality suppose $|S| \geq |T| = k$. Now, let $P = \cup_{i=1}^k\min(S^i,T^i) \cup \{S^{k+1},\dots,S^{|S|}\}$ and $Q = \cup_{i=1}^k\max(S^i,T^i)$. 
\par For $i < j \leq k$, $S^i < S^j$ and $T^i < T^j$. Hence, $\min(S^i,T^i) < \min(S^j,T^j)$ and $\max(S^i,T^i) < \max(S^j,T^j)$.
% \[ \min(S^i,T^i) \leq S^i, T^i < \min(S^j,T^j) \leq S^j,T^j  \]
% \begin{align*}
%     \min(S^i,T^i) &\leq S^i < S^j \leq \max(S^j, T^j) \text{, and} \\
%     \min(S^i,T^i) &\leq T^i < T^j \leq \max(S^j,T^j).
% \end{align*}
From this, we can deduce that $|P| = |S|$, $|Q| = |T|$, and for $i \in [k]$, $P^i = \min(S^i,T^i)$ and $Q^i = \max(S^i,T^i)$.
\par Clearly, $P \leq S,T \leq Q$ as $P^i \leq S^i,T^i \leq Q^i$ for all $i \in [k]$ and $P^i = S^i$ for $k < i \leq |S|$. Now, let $P',Q' \in 2^{[n]}$ such that $P' \leq S,T \leq Q'$, then,
\begin{align*}
    (P')^i &\leq \min(S^i,T^i) = P^i &&\text{for } i \in [k], \\
    (P')^i &\leq S^i = P^i &&\text{for } k < i \leq |S|, \\
    Q^i &= \max(S^i,T^i) \leq (Q')^i &&\text{for } i \in [k].
\end{align*}
Therefore, $P' \leq P$ and $Q \leq Q'$, and hence $P = S \wedge T$ and $Q = S \vee T$.
\end{proof}

\begin{corollary} \label{minmax}
Let $S,T \in 2^{[n]}$ such that $|S| = |T| = k$, then
\[ S \wedge T = \bigcup_{i=1}^k \min(S^i,T^i) \text{ and } S \vee T = \bigcup_{i=1}^k \max(S^i,T^i). \]
\end{corollary}

Throughout the rest of this paper, we will often take $N \in 2^{[n]}$ to be fixed and then consider $S \wedge N$ and $S \vee N$ for $S \in 2^{[n]}$, so we introduce the following functions. For each fixed $N \in 2^{[n]}$, define $\phi_N, \psi_N \colon 2^{[n]} \rightarrow 2^{[n]}$ by $\phi_N(S) = S \wedge N$ and $\psi_N(S) = S \vee N$.

Remark that $\phi_N$ and $\psi_N$ are order-preserving, that is, for $S,T \in 2^{[n]}$, if $S \leq T$, then $\phi_N(S) \leq \phi_N(T)$ and $\psi_N(S) \leq \psi_N(T)$. 
% Remark that $\phi_N(S)$ is the meet of $N$ and $S$ in the partial order and $\psi_N(S)$ is the join of $N$ and $S$. Hence, $\phi_N$ and $\psi_N$ are order preserving, that is, for $S,T \in 2^{[n]}_k$ we have that if $S \leq T$, then $\phi_N(S) \leq \phi_N(T)$ and $\psi_N(S) \leq \psi_N(T)$. 
\par Now, let $X = \rho_n(x)$ and $Y = \rho_n(y)$ for some $x,y \in [n]^*$. Consider $\cG_{X,Y}$, note that the rows of $X$ and $Y$ are indexed by $2^{[n]}$, so the nodes of $\cG_{X,Y}$ are subsets of $[n]$.
\par Let $e$ be an edge in $\cG_{X,Y}$ from $S$ to $T$ labelled $V \in \{X,Y\}$. We define $\phi_N(e)$ and $\psi_N(e)$ to be the edge from $\phi_N(S)$ to $\phi_N(T)$ labelled $V$ and the edge from $\psi_N(S)$ to $\psi_N(T)$ labelled $V$, respectively. Note that both these edges exist as $\phi_N$ and $\psi_N$ are order-preserving and by the definition of $\rho_n$, $(V)_{S,T} \neq -\infty$ if and only if $|S| = |T|$ and $S \leq T$.
\par Moreover, for any path $\gamma$ from $S$ to $T$ in $\cG_{X,Y}$, let $\phi_N(\gamma)$ and $\psi_N(\gamma)$ be the paths obtained by applying $\phi_N$ and $\psi_N$ respectively to the edges of $\gamma$. Again, $\phi_N(\gamma)$ and $\psi_N(\gamma)$ are well-defined paths in $\cG_{X,Y}$ as $\phi_N$ and $\psi_N$ are order-preserving.

\par Now, to proceed, we require a definition and lemma from \cite{JKPlactic}. We say that a word $w \in [n]^*$ is \emph{readable} from $S$ to $T$ if there exists an ordered sequence of sets
\[ S \leq P_1 \leq \cdots \leq P_{|w|} \leq T\]
such that $w_i \in P_i$ for each $1 \leq i \leq |w|$. 

\begin{lemma}[Lemma 2.2, \cite{JKPlactic}] \label{readlengthlem}
For every $w \in [n]^*$ and $S,T \in 2^{[n]}$ with $|S| = |T|$ and $S \leq T$, the
entry $\rho_n(w)_{S,T}$ is the maximum length of a scattered subword of $w$ that can be read
from $S$ to $T$
\end{lemma}

\begin{lemma} \label{splittingpaths}
Let $x,y \in [n]^*$, $X = \rho_n(x)$, $Y = \rho_n(y)$. Let $S,T \in 2^{[n]}$ such that $|S| = |T|$ and $S \leq T$, and $\gamma$ be a path from $S$ to $T$ in $\cG_{X,Y}$. Then, for any $N \in [S,T]$, there exist three paths all with the same labelling as $\gamma$, $\sigma$ from $S$ to $N$, $\tau$ from $N$ to $T$, and $\lambda$ from $N$ to $N$ such that $w(\gamma) + w(\lambda) \leq w(\sigma) + w(\tau)$.
\end{lemma}
\begin{proof}
Let $N \in [S,T]$, $\sigma = \phi_N(\gamma)$, and $\tau = \psi_N(\gamma)$, these paths have the same labelling at $\gamma$ by the definition of $\phi_N$ and $\psi_N$. Moreover, by Corollary~\ref{minmax}, $\sigma$ and $\tau$ are paths from $S$ to $N$ and $N$ to $T$ respectively as $S \leq N \leq T$. Let $\lambda$ be the path consisting entirely of loops at $N$ with the same labelling as $\gamma$. It suffices to show, that for each edge $f \in \gamma$,
\[w(f) + V_{N,N} \leq w(\phi_N(f)) + w(\psi_N(f)).\]
where $V \in \{X,Y\}$ is the matrix corresponding to the label of the edge $f$. Suppose $f$ is an edge from $F_1$ to $F_2$, $\phi_N(f)$ is an edge from $D_1$ to $D_2$ and $\psi_N(f)$ is an edge from $E_1$ to $E_2$. Then, this can be written as follows.
\begin{align*}
V_{F_1,F_2} + V_{N,N} \leq V_{D_1,D_2} + V_{E_1,E_2}. 
\end{align*}
If $V = X$, let $v = x$, otherwise, if $V = Y$, let $v = y$. Let $u$ be a scattered subword of $v$ that can be read from $F_1$ to $F_2$ of maximum length and $|u| = t$ for some $t \in \N_0$. Then, there exists an ordered sequence of sets
\[ F_1 \leq S_1 \leq \cdots \leq S_t \leq F_2\]
such that $u_i \in S_i$. Now, consider the two following sequences
\[ D_1 = \phi_N(F_1) \leq \phi_N(S_1) \leq \cdots \leq \phi_N(S_t) \leq \phi_N(F_2) = D_2,\]
\[ E_1 = \psi_N(F_1) \leq \psi_N(S_1) \leq \cdots \leq \psi_N(S_t) \leq \psi_N(F_2) = E_2.\]
By Corollary~\ref{minmax}, we have that for all $i \in [t]$, $u_i \in \phi_N(S_i) \cup \psi_N(S_i)$ and $N \subseteq \phi_N(S_i) \cup \psi_N(S_i)$. Moreover, if $u_i \in N$, then $u_i \in \phi_N(S_i) \cap \psi_N(S_i)$ as $u_i \in N \cap S_i$. Therefore, $u$ and every occurrence of each $m \in N$ is read by the above sequences from $D_1$ to $D_2$ and from $E_1$ to $E_2$. Hence, by Lemma~\ref{readlengthlem}, $V_{D_1,D_2} + V_{E_1,E_2} \geq  V_{F_1,F_2} + V_{N,N}$ as $V_{N,N} = |v|_N$. %is equal to the number of occurrences of elements from $N$ in $v$.
\end{proof}

\section{Plactic Identities} \label{PlacIdenChap}
In this section, we find a new set of semigroup identities satisfied by the plactic monoid of rank $n$. Moreover, we construct semigroup identities satisfied by $\Pb_n$ that are not satisfied by $UT_{n+1}(\T)$, showing that the variety generated by $UT_{n+1}(\T)$ is not contained in the variety generated by $\Pb_n$. 
Therefore, the variety generated by $\Pb_n$ is strictly contained in the variety generated by $\Pb_{n+1}$.

\begin{theorem} \label{SgpIdentityThm}
Let $q \in \{a,b\}^*$ contain every word of length $n-1$ in $\{a,b\}^*$ as a factor. Let $u = q^{h}aq^{h}$ and $v = q^{h}bq^{h}$ where $h = \floor{\frac{n^2}{4}}$. Then, the semigroup identity $u[a \mapsto ab, b \mapsto ba] = v[a \mapsto ab,b \mapsto ba]$ is satisfied by $\Pb_n$.
\end{theorem}
\begin{proof}
We need to show that, for all $X,Y \in \Pb_n$, 
\[u(a \mapsto XY, b \mapsto YX) = v(a \mapsto XY,b \mapsto YX). \]
So, let $x,y \in [n]^*$, $X = \rho_n(x)$, $Y = \rho_n(y)$ and for each $N \in 2^{[n]}$ say $w(N) = |xy|_N = (XY)_{N,N} = (YX)_{N,N}$ is the weight of $N$, where the second equality holds by Lemma~\ref{readlengthlem}. Note that $w(N)$ is a non-negative integer for all $N \in 2^{[n]}$. 
\par Recall, by the definition of $\rho_n$, for $S,T \in 2^{[n]}$, 
\begin{align*}
    u(a \mapsto XY, b \mapsto YX)_{S,T} \neq -\infty &\iff |S| = |T| \text{ and } S \leq T, \\
    &\iff v(a \mapsto XY, b \mapsto YX)_{S,T} \neq -\infty.
\end{align*}
Thus, to show $u(a \mapsto XY, b \mapsto YX) = v(a \mapsto XY, b \mapsto YX)$, it suffices to check the matrix entries indexed by $S, T$ where $|S| = |T|$ and $S \leq T$.
\par Let $S,T \in 2^{[n]}$ be such that $|S| = |T|$ and $S \leq T$, we aim to show that there exists a maximal weight path from $S$ to $T$ labelled by $u[a \mapsto XY, b \mapsto YX]$ in $\cG_{XY,YX}$ such that the central edge labelled by $XY$ is a loop. From this, we can construct a path of equal weight labelled by $v[a \mapsto XY,b \mapsto YX]$ by swapping the central edge labelled $XY$ with a loop labelled $YX$, hence showing that $u(a \mapsto XY, b \mapsto YX)_{S,T} \leq v(a \mapsto XY, b \mapsto YX)_{S,T}$.
\par Consider a maximal weight path $\gamma_1$ from $S$ to $T$ labelled $u[a \mapsto XY, b \mapsto YX]$. 
% For $i \in \N$, we now define $\gamma'_{i}$ and $\gamma_{i+1}$ recursively. For all $i \in \N$, let $\alpha_i$ and $\beta_i$ be the prefix and suffix of $\gamma_i$ labelled by $q^h[a \mapsto XY, b \mapsto YX]$ respectively, and let $A_i$ the highest weight node appearing latest in $\alpha_i$ and $B_i$ be the highest weight node appearing earliest in $\beta_i$.
For $i \in \N$, we now define $\gamma'_{i}$ and $\gamma_{i+1}$ recursively. For each $i \in \N$, write $\gamma_i = \alpha_i \circ \zeta_i \circ \beta_i$ where $\alpha_i$ and $\beta_i$ are labelled by $q^h[a \mapsto XY, b \mapsto YX]$ and $\zeta_i$ is labelled by $XY$. Let $A_i$ be the maximum weight node appearing latest in $\alpha_i$ and $B_i$ be the maximum weight node appearing earliest in $\beta_i$.
% \par Let $\gamma'_i$ be obtained taking $\gamma_i$ and removing any subpath of $\alpha_i$ or $\beta_i$ of length $|q|$ containing only loops and replacing each subpath in $\alpha_i$ removed with $|q|$ loops at $A_i$ and each subpath in $\beta_i$ removed with $|q|$ loops at $B_i$, labelling the loops so that $\gamma'_i$ is labelled $u[a \mapsto XY, b \mapsto YX]$. Note that all other edges retain their labelling as $\alpha_i$ and $\beta_i$ are labelled by $q^h[a \mapsto XY, b \mapsto YX]$, thus adding or removing $|q|$ consecutive loops shifts the edges $|q|$ places to the same label. Remark that $w(\gamma'_i) \geq w(\gamma_i)$ as $A_i$ and $B_i$ are highest weight nodes in $\alpha_i$ and $\beta_i$, respectively.
\par Let $\gamma'_i = \alpha_i' \circ \zeta_i \circ \beta_i'$, where $\alpha_i'$ and $\beta_i'$ are obtained by taking $\alpha_i$ and $\beta_i$ respectively and replacing any length $|q|$ subpath of $\alpha_i$ containing only loops with $|q|$ loops at $A_i$ and replacing any length $|q|$ subpath of $\beta_i$ containing only loops with $|q|$ loops at $B_i$, labelling the loops so that $\alpha_i'$ and $\beta'_i$ are labelled $q^h[a \mapsto XY, b \mapsto YX]$.
%where $\alpha_i'$ is obtained by taking $\alpha_i$ and removing any subpath of length $|q|$ containing only loops and replacing each with $|q|$ loops at $A_i$ labeling the loops so that $\alpha_i'$ is labelled by $q^h[a \mapsto XY, b \mapsto YX]$ and similarly $\beta_i'$ is obtained by taking $\beta_i$ and replacing each subpath of length $|q|$ containing only loops with $|q|$ loops at $B_i$, labelling the loops so that $\beta'_i$ is labelled $q^h[a \mapsto XY, b \mapsto YX]$. 
Note that all other edges retain their labelling as $\alpha_i$ and $\beta_i$ are labelled by $q^h[a \mapsto XY, b \mapsto YX]$, thus adding or removing $|q|$ consecutive loops shifts the edges $|q|$ places, to the same label. Remark that $w(\gamma'_i) \geq w(\gamma_i)$ as $A_i$ and $B_i$ are maximum weight nodes in $\alpha_i$ and $\beta_i$, respectively.
\par Clearly $A_i$ and $B_i$ are nodes of $\gamma'_i$. So, for $i \in \N$, let $\pi_i$ be the subpath of $\gamma'_i$ between $A_i$ and $B_i$ with no loops at $A_i$ or $B_i$, where $\pi_i$ is the empty path at $A_i$ if $A_i = B_i$. We define $\gamma_{i+1}$ in the following way.
\begin{enumerate}[(i)]
\item \label{caseA=B} If $A_i = B_i$, then let $\gamma_{i+1} = \gamma'_i$.

\item \label{caseAB} If $A_i \neq B_i$ and $w(A_i) \leq w(B_i)$. Then, by applying Corollary~\ref{splitapply}(\ref{splitapplyincre}) to $A_i$ and $B_i$, there exists $N_i \in [A_i,B_i]$, such that $[A_i,N_i]$ has chain length at most $n$ and for all $M \in [A_i,N_i]$,
\begin{equation} \label{mweightorder}
    w(M) \leq w(N_i).
\end{equation}
\par Now, apply Lemma~\ref{splittingpaths} to $\pi_i$ with $N_i$ to receive three paths, $\sigma_i$ from $A_i$ to $N_i$, $\tau_i$ from $N_i$ to $B_i$, and $\lambda_i$ from $N_i$ to $N_i$. Then,
\[ w(\pi_i) + |\lambda_i| w(N_i) = w(\pi_i) + w(\lambda_i) \leq w(\sigma_i) + w(\tau_i) \]
as $w(N_i) = (XY)_{N,N} = (YX)_{N,N}$, so $w(\lambda_i) = |\lambda_i| w(N_i)$.
\par Let $\sigma'_i$ be the maximal-length simple scattered subpath of $\sigma_i$. Note that, $[A_i,N_i]$ have chain length at most $n$, so $\sigma'_i$ has length at most $n-1$. By (\ref{mweightorder}), for $M \in [A_i,N_i]$, we have that $w(M) \leq w(N_i)$. Thus,
%\[ w(\pi_i) \leq w(\sigma_i') - |\sigma_i'|w(N_i) + w(\tau_i). \]
\begin{align*}
     w(\pi_i) &\leq w(\sigma_i) - |\lambda_i|w(N_i) + w(\tau_i) \\
     &\leq w(\sigma_i') - |\sigma_i'|w(N_i) + w(\tau_i)
\end{align*}
as $|\lambda_i| = |\sigma_i|$ and $w(\sigma_i) \leq w(\sigma_i') + ( |\sigma_i| -|\sigma_i'|)w(N_i)$. Therefore,
\begin{equation} \label{pii}
w(\pi_i) + |q|w(A_i) \leq w(\sigma_i') + (|q| - |\sigma_i'|)w(A_i) + w(\tau_i).
\end{equation}
as $w(A_i) \leq w(N_i)$.
\par As $q$ contains every word of length $n-1$ as a factor, there exists a factor of $q$, say $p_i$, such that $\sigma'_i$ is labelled by $p_i[a \mapsto XY,b \mapsto YX]$ as $|\sigma'_i|< n$. Write $q = s_ip_ie_i$ for some $s_i,e_i \in \{a,b\}^*$. Let $\delta_i$ be the path labelled $q[a \mapsto XY, b \mapsto YX]$ from $A_i$ to $N_i$ which contains $|s_i|$ loops at $A_i$ labelled $s_i[a \mapsto XY, b \mapsto YX]$, followed by the simple path $\sigma'_i$, and then $|e_i|$ loops at $N_i$ labelled $e_i[a \mapsto XY, b \mapsto YX]$.
\par The chain length of $[S,T]$ is at most $h+1$ by \cite[Theorem 3.2]{JKPlactic}. Thus, there are at most $h$ non-loop edges in $\gamma'_i$. As $A_i \neq B_i$, there is a non-loop in $\zeta_i \circ \beta_i'$. % the prefix labelled by $q^h[a \mapsto XY, b \mapsto YX]$. 
Hence, $\gamma'_i$ contains a subpath of loops at $A_i$ labelled by $q[a \mapsto XY, b \mapsto YX]$ as there are at most $h-1$ non-loop edges in $\alpha_i'$ % in the prefix labelled by $q^h[a \mapsto XY, b \mapsto YX]$ 
and any subpath of loops of length $|q|$ in $\alpha_i'$ %in the prefix 
is at $A_i$ by construction.
\par We construct $\gamma_{i+1}$ by first replacing the first $|q|$ loops at $A_i$ labelled by $q[a \mapsto XY,b \mapsto YX]$ with $\delta_i$, replacing $\pi_i$ with $\tau_i$, and replacing all the loops at $A_i$ in between with loops at $N_i$.

\item \label{caseBA} Similarly, if $w(A_i) > w(B_i)$. Then, by applying Corollary~\ref{splitapply}(\ref{splitapplydecre}) to $A_i$ and $B_i$, there exists $N_i \in [A_i,B_i]$, such that $[N_i,B_i]$ has chain length at most $n$ and for all $M \in [N_i,B_i]$,
\begin{equation} \label{mweightorder2}
    w(M) \leq w(N_i).
\end{equation}
\par Now, apply Lemma~\ref{splittingpaths} to $\pi_i$ with $N_i$ to receive three paths, $\sigma_i$ from $A_i$ to $N_i$, $\tau_i$ from $N_i$ to $B_i$, and $\lambda_i$ from $N_i$ to $N_i$. Then,
\[ w(\pi_i) + |\lambda_i| w(N_i) = w(\pi_i) + w(\lambda_i) \leq w(\sigma_i) + w(\tau_i) \]
as $w(N_i) = (XY)_{N,N} = (YX)_{N,N}$, so $w(\lambda_i) = |\lambda_i|w(N_i)$.
\par Let $\tau'_i$ be the maximal-length simple scattered subpath of $\tau_i$. Note that, $[N_i,B_i]$ have chain length at most $n$, so $\tau'_i$ has length at most $n-1$. By (\ref{mweightorder2}), for $M \in [N_i,B_i]$, we have that $w(M) \leq w(N_i)$. Thus,
%\[ w(\pi_i) \leq w(\sigma_i) + w(\tau_i') - |\tau_i'|w(N_i). \]
\begin{align*}
     w(\pi_i) &\leq w(\sigma_i) + w(\tau_i) - |\lambda_i|w(N_i)  \\
     &\leq w(\sigma_i) + w(\tau_i') - |\tau_i'|w(N_i)
\end{align*}
as $|\lambda_i| = |\tau_i|$ and $w(\tau_i) \leq w(\tau_i') + ( |\tau_i| -|\tau_i'|)w(N_i)$. Therefore,
\begin{equation} \label{pii2}
w(\pi_i) + |q|w(B_i) \leq w(\sigma_i) + w(\tau_i') + (|q| - |\tau_i'|)w(B_i).
\end{equation}
as $w(B_i) \leq w(N_i)$.
\par As $q$ contains every word of length $n-1$ as a factor, there exists a factor of $q$, say $p_i$, such that $\tau'_i$ is labelled by $p_i[a \mapsto XY,b \mapsto YX]$ as $|\tau'_i|< n$. Write $q = s_ip_ie_i$ for some $s_i,e_i \in \{a,b\}^*$. Let $\delta_i$ be the path labelled $q[a \mapsto XY, b \mapsto YX]$ from $N_i$ to $B_i$ which contains $|s_i|$ loops at $N_i$ labelled $s_i[a \mapsto XY, b \mapsto YX]$, followed by the simple path $\tau'_i$, and then $|e_i|$ loops at $B_i$ labelled $e_i[a \mapsto XY, b \mapsto YX]$.
\par The chain length of $[S,T]$ is at most $h+1$ by \cite[Theorem 3.2]{JKPlactic}. Thus, there are at most $h$ non-loop edges in $\gamma'_i$. As $A_i \neq B_i$, there is a non-loop in $\alpha_i' \circ \zeta_i$. %in the suffix labelled by $q^h[a \mapsto XY, b \mapsto YX]$. 
Hence, $\gamma'_i$ contains a subpath of loops at $B_i$ labelled by $q[a \mapsto XY, b \mapsto YX]$ as there are at most $h-1$ non-loop edges in $\beta_i'$ %in the suffix labelled by $q^h[a \mapsto XY, b \mapsto YX]$ 
and any subpath of loops of length $|q|$ in $\beta_i'$ %in the suffix 
is at $B_i$ by construction.
\par We construct $\gamma_{i+1}$ by first replacing $\pi_i$ with $\sigma_i$, replacing the last $|q|$ loops at $B_i$ labelled by $q[a \mapsto XY,b \mapsto YX]$ with $\delta_i$, and replacing all the loops at $B_i$ in between with loops at $N_i$. 
\end{enumerate}
\par In any of the above cases we have that $w(\gamma'_i) \leq w(\gamma_{i+1})$. In case (\ref{caseA=B}), this is clear, and in cases (\ref{caseAB}) and (\ref{caseBA}) we can see this by considering the construction of $\gamma_{i+1}$ with (\ref{pii}) and (\ref{pii2}) respectively. Thus, $\gamma_{i+1}$ is a maximal-weight path labelled $u[a \mapsto XY, b \mapsto YX]$ from $S$ to $T$.
\par If $A_i \neq B_i$ and $w(A_i) \leq w(B_i)$ then $N_i$ is a node of $\gamma_{i+1}$, $A_i < N_i$, and $w(A_i) \leq w(N_i)$. Thus, $A_i < A_{i+1}$ and $B_{i+1} \leq B_i$ as $B_i$ is a node of $\gamma_{i+1}$. Similarly, if $w(A_i) > w(B_i)$, then $A_i \leq A_{i+1}$ and $B_{i+1} < B_i$. If $A_i = B_i$ then by the definition of $\gamma_{i+1}$, $A_i = A_{i+1}$ and $B_i = B_{i+1}$.
\par So, if $A_i \neq B_i$, then the chain length of $[A_{i+1},B_{i+1}]$ is strictly less than the chain length of $[A_i,B_i]$. Thus, as $[S,T]$ has finite chain length, we have that there exists $k \in \N$ such that $A_k = B_k$.
\par Therefore, $\gamma_k = \alpha_k \circ \zeta_k \circ \beta_k$ is a maximal-weighted path labelled $u[a \mapsto XY, b \mapsto YX]$ such that $\zeta_k$ is a loop at $A_k = B_k$ labelled $XY$, and hence $u(a \mapsto XY, b \mapsto YX)_{S,T} \leq v(a \mapsto XY, b \mapsto YX)_{S,T}$.
\par Moreover, we can dually show that there exists a maximal weighted path labelled $v[a \mapsto XY, b \mapsto YX]$ such that the middle edge is a loop labelled $YX$ implying that $u(a \mapsto XY, b \mapsto YX)_{S,T} \geq v(a \mapsto XY, b \mapsto YX)_{S,T}$. 
Thus, as $x,y \in [n]^*$ were taken arbitrarily and $\rho_n$ is a faithful representation of $\Pb_n$, the semigroup identity $u[a \mapsto ab, b \mapsto ba] = v[a \mapsto ab,b \mapsto ba]$ is satisfied by $\Pb_n$.
\end{proof}

\begin{remark}
The shortest semigroup identities satisfied by $\Pb_n$ given by the above theorem are of length $O(2^{n-1}n^2)$ as the shortest word containing every word of length $n-1$ in $\{a,b\}^*$ as a factor is of length $O(2^{n-1})$. \cite[p.79]{RibeiroThesis} The shortest previously known identities for $\Pb_n$ are of length at least $O(2^{h+2})$ where $h = \floor{\frac{n^2}{4}}$. Thus, the identities given in this paper are shorter than those previously known for all $n \geq 6$. When $n = 6$, previously the shortest known identity was of length 2082, %2^{n+1} + 4n -6 when n = 10
now the shortest known identity is of length 1298. %4(floor(n^2/4))(2^{n-1} + n -2) + 2 when n = 6
\end{remark}

\begin{corollary} \label{PlacnUTn1}
Let $n \in \N$. There exists a semigroup identity satisfied by $\Pb_n$ but not by $UT_{n+1}(\T)$.
\end{corollary}
\begin{proof}
Let $q \in \{a,b\}^*$, begin with $b$, end in $a^{n-1}$, contain every word of length $n-1$ in $\{a,b\}^*$ as a factor, but not contain $a^n$ as a factor. Let $u = q^{h}aq^{h}$ and $v = q^{h}bq^{h}$ where $h = \floor{\frac{n^2}{4}}$. Then, by Theorem~\ref{SgpIdentityThm}, $u[a \mapsto ab,b \mapsto ba] = v[a \mapsto ab, b \mapsto ba]$ is satisfied by $\Pb_n$. However, by \cite[Lemma 3.1]{ATropicalPlactic}, $u[a \mapsto ab,b \mapsto ba] = v[a \mapsto ab, b \mapsto ba]$ is not satisfied by $UT_{n+1}(\T)$ as $a^n$ is a factor of $u$ but not $v$.
\end{proof}

We can now conclude that plactic monoids of different ranks generate different semigroup varieties.
\begin{corollary} \label{PlacDiff}
Let $n \in \N$. There exists a semigroup identity satisfied by $\Pb_n$ but not satisfied by $\Pb_{n+1}$.
\end{corollary}
\begin{proof}
By Corollary~\ref{PlacnUTn1}, there exists a semigroup identity satisfied by $\Pb_n$ but not $UT_{n+1}(\T)$. However, all identities satisfied by $\Pb_{n+1}$ are satisfied by $UT_{n+1}(\T)$ \cite[Theorem 4.4]{JKPlactic}. Thus, there exists an identity satisfied by $\Pb_n$ which is not satisfied by $\Pb_{n+1}$.
\end{proof}

\begin{corollary}
Let $n \in \N$. If $\Pb_n$ generates the same semigroup variety as $UT_k(\T)$, then $k = n$.
\end{corollary}
\begin{proof}
By \cite[Theorem 3.4]{ATropicalPlactic}, for all $k < n$, there exists an identity satisfied by $UT_{k}(\T)$ but not by $UT_n(\T)$, but by \cite[Theorem 4.4]{JKPlactic}, every identity satisfied by $\Pb_n$ is satisfied by $UT_n(\T)$. Thus, for all $k < n$, there exists an identity satisfied by $UT_k(\T)$ but not $\Pb_n$. Finally, by Corollary~\ref{PlacnUTn1}, there exists an identity satisfied by $\Pb_n$ but not $UT_k(\T)$ for all $k > n$.
\end{proof}

It is known that the plactic monoid of rank $n$ generates the same semigroup variety as $UT_n(\T)$ for $n \leq 3$, so given the above corollary, we pose the following question.
\begin{question}
Does $\Pb_n$ generate the same semigroup variety as $UT_n(\T)$ for any $n \geq 4$?
\end{question}

\bibliographystyle{plain}
\bibliography{Bib.bib}
\end{document}